\documentclass{elsarticle}
\usepackage[english]{babel}
\usepackage{amsmath, amssymb, amsthm}
\usepackage{amsfonts}

\newtheorem{thm}{Theorem}[section]

\newtheorem{lem}{Lemma}[section]

\newtheorem{conj}{Conjecture}[section]
\newtheorem{prob}{Problem}[section]
\theoremstyle{definition}

\theoremstyle{remark}

\numberwithin{equation}{section}

\begin{document}

\begin{frontmatter}
	
\title{On the Ilmonen-Haukkanen-Merikoski Conjecture}

\author[gazi]{Ercan Alt\i n\i \c{s}\i k \corref{corresponding}}
\ead{ealtinisik@gazi.edu.tr}
\author[gazi]{Ali Keskin }
	\ead{akeskin1729@gmail.com}
\author[gazi]{Mehmet Y\i ld\i z }
\ead{yildizm78@mynet.com}
\author[bilkent]{Murat Demirb\"{u}ken }
\ead{murdem91@gmail.com}

\cortext[corresponding]{Corresponding author. Tel.: +90 312 202 1070}
\address[gazi]{Department of Mathematics, Faculty of Sciences, Gazi University \\ 06500 Teknikokullar - Ankara, Turkey}
\address[bilkent]{Department of Computer Science, Faculty of Engineering, \.{I}. Do\u{g}ramac\i \ Bilkent University \\ 06800 Bilkent - Ankara, Turkey}

\begin{abstract}
Let $K_n$ be the set of all $n\times n$ lower triangular (0,1)-matrices with each diagonal element equal to $1$, $L_n = \{ YY^T: Y\in K_n\}$ and let 
\begin{equation*}
c_n = \min_{Z\in L_n} \left\lbrace \mu_n^{(1)}(Z):\mu_n^{(1)} (Z) \text{ is the smallest eigenvalue of } Z \right\rbrace.
\end{equation*}
The Ilmonen-Haukkanen-Merikoski conjecture (the IHM conjecture) states that $c_n$ is equal to the smallest eigenvalue of $Y_0Y_0^T$, where 
\begin{equation*}
	(Y_0)_{ij}=\left\lbrace \ \begin{array}{cl}
		0 & \text{if } \ i<j, \\ 
		1 & \text{if } \ i=j,\\
		\frac{1-(-1)^{i+j}}{2} & \text{if } \ i>j.  
	\end{array}
	\right.
\end{equation*}  
In this paper we present a proof of this conjecture. In our proof we use an inequality for spectral radii of nonnegative matrices. 
\end{abstract}


\end{frontmatter}

\section{Introduction and Statement of the Main Theorem}
Let $S=\{x_1,x_2,\ldots,x_n\}$ be a set of distinct positive integers, $(x_i,x_j)$ denote the greatest common divisor of $x_i$ and $x_j$ and let $\varepsilon$ be a positive real number. The $n \times n$ matrices $(S)=((x_i,x_j))$ and $(S^\varepsilon)=((x_i,x_j)^\varepsilon)$ are called the GCD matrix and the power GCD matrix on $S$, respectively. In 1876, Smith \cite{Smith} proved that if $S$ is factor closed, then
$\det(S)=\prod_{k=1}^n \varphi(x_k)$, where $\varphi$ is Euler's totient. Since
then many results on these matrices have been published in the
literature, see e.g. \cite{Altinisik2005, Altinisik2004,BeslinLigh,BourqueLigh1992,Hauk1997,KorkeeHauk2003}. 

One of interesting and active area in the study of GCD matrices is their eigenstructure. The first results on this subject were published in the papers \cite{Wintner,Lindqvist}. Since these results are consequences of some theorems on Riesz bases in functional analysis, the paper of Hong and Loewy \cite{HongLoewy2004} can be considered as the first paper on the study of the eigenvalues of GCD and related matrices due to the number theoretical aspect of the subject. Since their pioneering paper many results on the subject have been published in the literature, see e.g. \cite{Altinisik2009,AltinisikBuyukkose2014,AltinisikBuyukkose2015,Hong2008,HongEnochLee,HongLoewy2011,IlmonenHaukkanenMerikoski,Mattila2015,Mattila2012,Mattila2012paper,MattilaHauk2014}. 
In that paper Hong and Loewy investigated the asymptotic behavior of the eigenvalues of power GCD matrices by using some tools of number theory. Beside their results on asymptotic behavior of these matrices, in the same paper Hong and Loewy introduced a constant $c_n$ and used it to present a lower bound for the smallest eigenvalues of power GCD matrices. Let $K_n$ be the set of all $n\times n$ lower triangular (0, 1)-matrices which each diagonal element equal to $1$ and let $L_n = \{ YY^T: Y\in K_n\}$. They defined $c_n$ depending only on $n$ as follows:
\begin{equation} \label{constantc}
	c_n = \min_{Z\in L_n} \left\lbrace \mu_n^{(1)}(Z):\mu_n^{(1)} (Z) \text{ is the smallest eigenvalue of } Z \right\rbrace.
\end{equation}
Then they prove that 
\begin{equation*}
	\lambda_n^{(1)} ((S^\varepsilon)) \geq c_n \cdot \min_{1\leq i \leq n} \{ J_\varepsilon(x_i) \},
\end{equation*}
where $J_\varepsilon$ is Jordan's generalization of Euler's totient and $\lambda_n^{(1)} ((S^\varepsilon))$ is the smallest eigenvalue of the power GCD matrix $(S^\varepsilon)$, see \cite[Theorem~4.2]{HongLoewy2004}.

In 2008, in the light of their MATLAB calculations for $n=2,3,\ldots,7$, Ilmonen, Haukkanen and Merikoski \cite{IlmonenHaukkanenMerikoski} presented an interesting conjecture about the constant $c_n$.  
\begin{conj} \label{conjecture} [The IHM conjecture, see \cite{IlmonenHaukkanenMerikoski}, Conjecture 7.1]
Let $Y_0 =(b_{ij}) \in M_n(\mathbb{R}) $  be defined by 		
	\begin{equation} \label{Y0Matrix}
	b_{ij}=\left\lbrace \ \begin{array}{cl}
	0 & \text{if } \ i<j, \\ 
	1 & \text{if } \ i=j,\\
	\frac{1-(-1)^{i+j}}{2} & \text{if } \ i>j.  
	\end{array}
	\right.
	\end{equation}
Then $c_n$ is equal to the smallest eigenvalue of $Y_0Y^T_0.$
\end{conj}

Numerical evidence of the IHM conjecture as follows. Recently, the first, the second and the fourth author of the paper have investigated that the IHM conjecture holds for $n=7$ and $n=8$ with the help of a MATLAB code. 
Our MATLAB code running on a computer\footnote{Intel Core i7-920 Quad Core 2.66 GHz 8 MB L Cache 24 GB DDR3 RAM} has verified the truth of the IHM conjecture for $n=7$ in 23 minutes and for $n=8$ in 3.5 days. Since $|L_8|=2^{28}=268,435,456$ 
and $|L_9|=2^{36}=68,719,476,736$, it would take about $3$ years to verify the IHM conjecture for $n=9$ with the help of our MATLAB code. To overcome this difficulty about time, we write a different code in C programming language. We use Newton's identities (see \cite{Kalman}) to obtain the characteristic polynomial of a matrix $Z$ in $L_n$ and we calculate the smallest eigenvalue of $Z$ by using Newton's method (see \cite{SuliMayers}) to shorten the running time. Indeed, our C code running on the same computer has verified the truth of the IHM conjecture in 30 minutes for $n=8$ and in 7 days for $n=9$. Thus, we have concluded that Conjecture~\ref{conjecture} holds for $n=8$ and $n=9$. This investigation has been presented by the first author of this paper in \cite{Altinisik2015}.

After obtaining enough numerical evidence that the IHM conjecture can be true, we get the motivation to find out a proof of it. 

\begin{thm}[The Main Theorem] \label{conjproof}
	Assume the setup above. Then the IHM conjecture is true.
\end{thm}	

The strategy of the proof as follows. We prove that for any matrix $Y$ in $K_n$, $|Y^{-1}|\leq |Y_0^{-1}|$, where the matrix $Y_0$ given by (\ref{Y0Matrix}). Secondly, we show that $|Z^{-1}|\leq |(Y_0Y_0^T)^{-1}|$ for all $Z\in L_n$. Then, by using an inequality for spectral radii of nonnegative matrices, we obtain a proof of the IHM conjecture. We conclude the paper with a discussion about further studies on the constant $c_n$ and a conjecture on the uniqueness of the matrix $Y_0$.

\section{Proof of the Main Theorem}

First we present a simple fact about a particular nilpotent (0,1)-matrix which we use in the course of our proofs. Here we give the proof of this fact  
though one can find in the literature.
\begin{lem} \label{NLemma}
	Let $Y\in K_n$ and $N:=Y-I$, where $I$ is the $n\times n$ identity matrix. We denote by $(N^k)_{ij}$ the $ij-$entry of the positive integer $k$-th power of $N$. Then we have the following
	\begin{description}
		\item [i)] $(N^k)_{ij}=0$ whenever $i-j<k$,
		\item [ii)] $Y^{-1}=I-N+N^2-\cdots +(-1)^{n-1}N^{n-1}$.
	\end{description}
	
\end{lem}

\begin{proof} Since it is clear that $N$ is a strictly lower triangular (0,1)-matrix, the proof of the first claim follows from the matrix multiplication.
	For the proof of the second claim, consider
	$$I^n-(-N)^n=(I+N)(I-N+N^2-\cdots +(-1)^{n-1}N^{n-1}).$$
	Since $N^n=0$ and $Y=I+N$, we have
	$$Y^{-1}=I-N+N^2-\cdots +(-1)^{n-1}N^{n-1}.$$
\end{proof}	

Now we investigate the inverse of any matrix $Y$ in $K_n$. In the following lemma we obtain a recurrence relation for the entries of the inverse of $Y$. 
\begin{lem} \label{recurrenceLemma}
	Let $Y\in K_n$ and $N:=Y-I$. Also, let $N=(n_{ij})$ and $Y^{-1}=(a_{ij})$. Then we have the following recurrence relation for $a_{kl}:$
	$$
	a_{kl}=\left\{
	\begin{array}{cl}
	0 & \text{ if }  \ k<l,\\
	1 & \text{ if }  \ k=l,\\
	-\sum_{i=l}^{k-1}n_{ki}a_{il} & \text{ if }  \ k>l.
	\end{array}
	\right.
	$$
\end{lem}

\begin{proof}
	When we multiply both sides of the equality in Lemma~\ref{NLemma}.(ii) from the left by $-N$, we have
	$$
	-NY^{-1}=-N+N^2-\cdots +(-1)^{n-1}N^{n-1}+(-1)^nN^n.
	$$
	Since $N^n=0$, one can easily obtain
	\begin{equation}\label{inverseY}
	I-NY^{-1}=Y^{-1}. 
	\end{equation}
	By Lemma~\ref{NLemma}, it is clear that $a_{kl}=0$ if $k<l$ and $a_{kl}=1$ if $k=l$. Now, from (\ref{inverseY}), we have
	$$
	a_{kl}=-\sum_{i=1}^{n}n_{ki}a_{il}
	$$
	and hence, by Lemma~\ref{NLemma}, we obtain
	$$
	a_{kl}=-\sum_{i=l}^{k-1}n_{ki}a_{il}
	$$
	for all $k>l$.
\end{proof}
In the following theorem, we find the largest absolute value of each $a_{ij}$ in terms of Fibonacci numbers which is a surprising result.
  
\begin{thm} \label{maxTheorem}
	Let $Y\in K_n$ and $Y^{-1}=(a_{ij})$. Then, for $1\leq j < i\leq n$, we have $\left|a_{ij} \right|\leq F_{i-j} $, where $F_{i-j}$ is the $(i-j)$-th Fibonacci number. 
\end{thm}
\begin{proof} Let $N=(n_{ij})$ be as in Lemma~\ref{recurrenceLemma}. Also, let $1\leq j <i\leq n$ and $t=i-j$. We will prove that $\left|a_{j+t,j}  \right|\leq F_{t} $ by induction on $t$. 
Let $t=1$. By Lemma~\ref{recurrenceLemma}, we have $\left|a_{j+1,j}  \right| = n_{j+1,j}$, where $n_{j+1,j}$ can be $0$ or $1$. Thus, $\left|a_{j+1,j}  \right| \leq 1 = F_{1}$. Now assume that $\left|a_{j+t,j}  \right|\leq F_{t} $ for each $t=1,2,\ldots , k-1$. By Lemma~\ref{recurrenceLemma}, we have $a_{j+k,j}=-\sum_{i=j}^{j+k-1}n_{j+k,i}a_{ij}$. 

Case~1. Assume $n_{j+k,j+k-1}=0$. Then $\left|a_{j+k,j}  \right|=\left| \sum_{i=j}^{j+k-2}n_{j+k,i}a_{ij}\right| $. Also, by Lemma~\ref{recurrenceLemma}, $a_{j+k-1,j}=-\sum_{i=j}^{j+k-2}n_{j+k-1,i}a_{ij}$. Since both of $n_{j+k,i} $ and $n_{j+k-1,i}$ for each $i=j,j+1,\ldots,j+k-2$ can arbitrarily be $0$ or $1$, it is clear that $a_{j+k,j}$ and $a_{j+k-1,j}$ have the same values. Therefore, by the induction hypothesis, we obtain $\left|a_{j+k,j}  \right| \leq F_{k-1} \leq F_k$. 

Case~2. Assume $n_{j+k,j+k-1}=1$.

Subcase~i. Assume $n_{j+k-1,j+k-2}=0$. By Lemma~\ref{recurrenceLemma}, we have
$$
\left| a_{j+k,j} \right| \leq \left| \sum_{i=j}^{j+k-2}n_{j+k,i}a_{ij} \right| + \left| a_{j+k-1,j} \right|. 
$$
Also, by Lemma~\ref{recurrenceLemma}, it is clear that $a_{j+k-1,j}=-\sum_{i=j}^{j+k-2}n_{j+k-1,i}a_{ij}$. Since both of $n_{j+k,i} $ and $n_{j+k-1,i}$ for each $i=j,j+1,\ldots,j+k-2$ can arbitrarily be $0$ or $1$, it is clear that $\sum_{i=j}^{j+k-2}n_{j+k,i}a_{ij}$ and $a_{j+k-1,j}$ have the same values. Thus, by the induction hypothesis, we obtain $\left| \sum_{i=j}^{j+k-2}n_{j+k,i}a_{ij} \right| \leq F_{k-1}$. Beside this, by our assumption in Subcase~i, $\left| a_{j+k-1,j} \right| = \left| \sum_{i=j}^{j+k-3}n_{j+k-1,i}a_{ij} \right| $. Since both of $n_{j+k-1,i} $ and $n_{j+k-2,i}$ for each $i=j,j+1,\ldots,j+k-3$ can arbitrarily be $0$ or $1$, it is obvious that $\sum_{i=j}^{j+k-3}n_{j+k-1,i}a_{ij}$ and $a_{j+k-2,j}$ have the same values. By the induction hypothesis, we obtain $\left| a_{j+k-1,j}\right|  \leq F_{k-2}.$ Thus, $\left| a_{j+k,j} \right| \leq F_{k-1} + F_{k-2}=F_{k}.$

Subcase~ii. Assume $n_{j+k-1,j+k-2}=1$. By Lemma~\ref{recurrenceLemma}, we have
$$
\left| a_{j+k,j} \right| \leq \left| \sum_{i=j}^{j+k-3}n_{j+k,i}a_{ij} \right| + \left| \sum_{i=j+k-2}^{j+k-1}n_{j+k,i}a_{ij} \right|. 
$$
Since  $\sum_{i=j}^{j+k-3}n_{j+k,i}a_{ij}$ and $a_{j+k-2,j}$ have the same values, by the induction hypothesis, we have $\left| \sum_{i=j}^{j+k-3}n_{j+k,i}a_{ij}\right|  \leq F_{k-2}$. In addition to this, since $n_{j+k-1,j+k-2}=1$ we obtain 
$$
\sum_{i=j+k-2}^{j+k-1}n_{j+k,i}a_{ij}= (n_{j+k,j+k-2}-1)a_{j+k-2,j}-\sum_{i=j}^{j+k-3}n_{j+k-1,i}a_{ij}.
$$
Here each $(1-n_{j+k,j+k-2}), n_{j+k-1,j}, \ldots, n_{j+k-1,j+k-3}$ can arbitrarily be $0$ or $1$. Thus, $\sum_{i=j+k-2}^{j+k-1}n_{j+k,i}a_{ij}$ and $a_{j+k-1,j}$ have the same values and hence, by the induction hypothesis,  $\left| \sum_{i=j+k-2}^{j+k-1}n_{j+k,i}a_{ij}\right| \leq F_{k-1} $. Therefore, we obtain $\left|  a_{j+k,j}\right|  \leq F_{k-2} + F_{k-1} = F_k$.

The principle of induction completes the proof. 
\end{proof}

Let $A=(a_{ij})$, $B=(b_{ij})\in M_n(\mathbb{R}),$ that is, the set of all $n \times n$ real matrices. We write $A\geq 0$ (or $> 0$) if all $a_{ij}\geq 0$ (or $> 0$).
Also, we write $A\geq B$ (or $>0$) if $A-B\geq 0$ (or $>0$).
In addition to this, we define $\left| A\right|=(\left| a_{ij}\right| )$, that is, $\left| A\right| $  is the element-wise absolute value of $A$. The largest eigenvalue of $A$ in modulus is denoted by $\rho (A)$ and called the spectral radius of $A$. Now we fix the notation for the rest of the paper.
\begin{thm} \label{Last}
	Let $Y_0=(b_{ij})$ be as in (\ref{Y0Matrix}) and let $Z_0:=Y_0Y_0^T$. For all $Z\in L_n$, we have $|Z^{-1}|\leq |Z_0^{-1}|$.
\end{thm}
\begin{proof}
	First we obtain the inverse of $Y_0$. Let $N_0=Y_0-I$ and $N_0=(m_{ij})$. Then
	$$
	m_{ij}=\left\{
	\begin{array}{cl}
	0 & \text{ if }  \ i\leq j,\\
	\frac{1-(-1)^{i+j}}{2} & \text{ otherwise. }
	\end{array}
	\right.
	$$
	We claim that the inverse of $Y_0$ is the $n\times n$ matrix $(c_{ij})$, where
	$$
	c_{ij}=\left\{
	\begin{array}{cl}
	0 & \text{ if }  \ i<j,\\
	1 & \text{ if }  \ i=j,\\
	(-1)^{i-j}F_{i-j} & \text{ if }  \ i>j.
	\end{array}
	\right.
	$$
	Since $Y_0\in K_n$ by Lemma~\ref{recurrenceLemma},
	it is clear that $c_{ij}=0$ if $i<j$ and $c_{ij}=1$ if $i=j$.
	Also, by Lemma~\ref{recurrenceLemma}, we have
	$$
	c_{ij}=-\sum_{k=j}^{i-1}m_{ik}c_{kj}
	$$
	for $i>j$.
	Now we prove that $c_{ij}=(-1)^{i-j}F_{i-j}$ whenever $i>j$ by induction on $t=i-j$.
	For $t=1$,  $$c_{j+1,j}=-m_{j+1,j}=-1=-F_1.$$
	Assume that $c_{j+t,j}=(-1)^tF_t$ for all $t=1,2,\ldots,k-1$.
	Recall that
	$$
	c_{j+k,j}=-\sum_{s=j}^{j+k-1}m_{j+k,s}c_{sj}.
	$$
	By the induction hypothesis, if $k$ is even then we have
	$$
	c_{j+k,j}=\sum_{s=1}^{k/2}F_{2s-1} =F_k
	$$
	and if $k$ is odd then
	$$
	c_{j+k,j}=-1-\sum_{s=1}^{(k-1)/2}F_{2s} =-F_k. 
	$$
	Thus, $c_{j+k,j}=(-1)^kF_k$. 
	
	Secondly, we calculate the inverse of $Z_0$.
	Since $Z_0=Y_0Y_0^T$ and $Y_0^{-1}=(c_{ij})$, we have
	$$
	(Z_0^{-1})_{ii}=\sum_{k=1}^{n}c_{ki}^2=1+\sum_{k=i+1}^{n}F_{k-i}^2
	$$
	for all $i=1,2,\ldots,n$.
	Now let $1\leq i<j\leq n$. Then
	\begin{eqnarray*}
		(Z_0^{-1})_{ij}&=&\sum_{t=1}^{n}c_{ti}c_{tj} \\
		&=&c_{ji}+\sum_{t=j+1}^{n}c_{ti}c_{tj} \\
		&=&(-1)^{j-i}F_{j-i}+\sum_{t=j+1}^{n}(-1)^{-i-j}F_{t-i}F_{t-j} \\
		&=&(-1)^{j-i}(F_{j-i}+\sum_{t=j+1}^{n}F_{t-i}F_{t-j}). \\
	\end{eqnarray*}
	Since $Z_0^{-1}$ is symmetric, for all $1\leq j<i\leq n$,
	$$
	(Z_0^{-1})_{ij}=(-1)^{i-j}(F_{i-j}+\sum_{t=i+1}^{n}F_{t-i}F_{t-j}).
	$$
	Now we prove the claim of the theorem. For each $Z\in L_n$, there exists a matrix $Y$ in $K_n$ such that $Z=YY^T$. Let $Y^{-1}=(a_{ij})$. Then, by Lemma~\ref{recurrenceLemma} and Theorem~\ref{maxTheorem}, we have 
	\begin{eqnarray*}
		\left| (Z^{-1})_{ii}\right|  &=&\left|\sum_{k=1}^{n}a_{ki}^2\right| \\
		&=& \sum_{k=1}^{n}|a_{ki}|^2\\
		&=& \sum_{k=i}^{n}|a_{ki}|^2 \\
		&\leq& 1+ \sum_{k=i+1}^{n}F_{k-i}^2 \\
		&=&\left| (Z_0^{-1})_{ii}\right|
	\end{eqnarray*}
	for all $i=1,2,\ldots,n$.
	
	Let $j>i$. By Lemma~\ref{recurrenceLemma} and Theorem~\ref{maxTheorem}, we have
	\begin{eqnarray*}
	\left| (Z^{-1})_{ij} \right| 	&=& \left|\sum_{t=1}^{n}a_{ti}a_{tj}\right| \\
		&\leq&\sum_{t=1}^{n}|a_{ti}||a_{tj}| \\
		&=& |a_{ji}|+\sum_{t=j+1}^{n}|a_{ti}||a_{tj}|\\
		&\leq& F_{j-i}+\sum_{t=j+1}^{n}F_{t-i}F_{t-j}\\
		&=&\left| (Z_0^{-1})_{ij}\right| 
	\end{eqnarray*}
	Finally, since $Z^{-1}$ and $Z_0^{-1}$ are symmetric $|Z^{-1}|\leq |Z_0^{-1}|$ for all $Z\in L_n$.
	
\end{proof}	
The following lemma is crucial in proof of the Main Theorem.
\begin{lem}[See \cite{HornJohnson}, Theorem 8.1.18] \label{HornJohnsonLemma}
	Let $A,B \in M_n\mathbb{(R)}$. If $\left| A\right| \leq B$, then $\rho (A)\leq \rho (\left|A\right|)\leq\rho (B)$. 
\end{lem}

Finally, 
we are ready to give proof of the Main Theorem.

\begin{proof}[\textbf{Proof of the Main Theorem}]
	Let $Z_0$ be as in Theorem~\ref{Last}. First we prove that the matrices $Z_0^{-1}$ and $|Z_0^{-1}|$ have the same characteristic polynomial. By the definition of the trace of a square matrix, it is clear that
	$$ 
	trace((Z_0^{-1})^k)=\sum_{i_1,\ldots , i_k=1}^{n}(Z_0^{-1})_{i_1i_2}\ldots (Z_0^{-1})_{i_{k-1}i_k}(Z_0^{-1})_{i_ki_1}
	$$
	for each $k=1,2,\ldots ,n$. Also, from the formulae for $(Z_0^{-1})_{ij}$ in the proof of Theorem~\ref{Last}, one can easily show that $sgn(Z_0^{-1})_{ij}=(-1)^{i-j}$ for all $1\leq i,j\leq n$. Thus, we have $$sgn((Z_0^{-1})_{i_1i_2}\ldots (Z_0^{-1})_{i_{k-1}i_k}(Z_0^{-1})_{i_ki_1})=1$$ and hence $trace(|Z_0^{-1}|^k) = trace((Z_0^{-1})^k)$. By Newton's identities \cite{Kalman}, we obtain that $Z_0^{-1}$ and $|Z_0^{-1}|$ have the same characteristic polynomial. Thus, $\rho (|Z_0^{-1}|)=\rho (Z_0^{-1})$. From Theorem~\ref{Last} and Lemma~\ref{HornJohnsonLemma}, now we obtain
$$
\rho (Z^{-1})\leq \rho (|Z^{-1}|)\leq \rho (|Z_0^{-1}|)=\rho (Z_0^{-1}).
$$
Since all $Z$ in $L_n$ are positive definite, the smallest eigenvalue of $Z_0$ is less than or equal to the smallest eigenvalue of $Z$ for all $Z$ in $L_n$. \end{proof}	

\section{Some Results and Open Problems on the Constant $c_n$}

In the literature there are not so many results on estimating the value of $c_n$. Recently, Mattila \cite{Mattila2015} has presented a lower bound for $c_n$. Indeed, he proved that $c_n$ is bounded below by $( \frac{6}{n^4+2n^3+2n^2+n} )^{\frac{n-1}{2}} $. Then he showed that this lower bound can be replaced with $ ( \frac{48}{n^4+56n^2+48n} )^{\frac{n-1}{2}} $ when $n$ is even, and $ ( \frac{48}{n^4+50n^2+48n-51} )^{\frac{n-1}{2}} $ when $n$ is odd. Recently, beside Mattila's results,  Alt\i n\i \c{s}\i k and B\"{u}y\"{u}kk\"{o}se \cite{AltinisikBuyukkose2015} have obtained a lower bound for the smallest eigenvalue $t_n$ of the $n\times n$ matrix $E_n^T E_n$, where the $ij-$ entry of $E_n$ is 1 if $j|i$ and $0$ otherwise, i.e., $t_n \geq \left(n\sum_{k=1}^{n} \mu^2(k) \right)^{-1}$. Indeed, this bound can be used instead of lower bounds including $c_n$ for the smallest eigenvalues of GCD and related matrices defined on $S=\{1,2,\ldots,n\}$ in the literature, see \cite{Hong2008,HongLoewy2004,IlmonenHaukkanenMerikoski,Mattila2015,MattilaHauk2014}. After above studies on estimating the value of $c_n$, we naturally raise the following problem.

\begin{prob}
Can one improve the lower bounds mentioned above for $c_n$? 
\end{prob}

On the other hand, in our investigation \cite{Altinisik2015}, we cannot find any matrix $Y$ other than $Y_0$ in $K_n$ such that $c_n$ is equal to the smallest eigenvalue of $YY^T$ for each $n=2,3,\ldots , 9$. After this observation we can present the following conjecture. 

\begin{conj}
	Let $n\geq 2$. There is a unique matrix $Y$ in $K_n$ such that $c_n$ is equal to the smallest eigenvalue of $YY^T$. In other words, if the smallest eigenvalue of $YY^T$ is equal to $c_n$ then $Y=Y_0$, where $Y_0$ is defined by (\ref{Y0Matrix}). 
\end{conj}


\begin{thebibliography}{9}


\bibitem{Altinisik2009} E. Alt\i n\i \c{s}\i k, On inverses of GCD matrices associated with multiplicative functions and a proof of the Hong-Loewy conjecture, Linear Algebra Appl. 430 (2009) 1313-1327.

\bibitem{Altinisik2015} E. Alt\i n\i \c{s}\i k, On a Conjecture on the Smallest Eigenvalues of Some Special Positive Definite Matrices, 3rd International Conference on Applied Mathematics \& Approximation Theory - AMAT 2015, 28-31 May 2015, Ankara, Turkey, 2015.


\bibitem{Altinisik2005} E. Alt\i n\i \c{s}\i k, B. E. Sagan and N. Tu\~{g}lu, GCD matrices, posets, and
nonintersecting paths, Linear and Multilinear Algebra 53(2) (2005) 75-84.

\bibitem{AltinisikBuyukkose2014} E. Alt\i n\i \c{s}\i k and \c{S}. B\"{u}y\"{u}kk\"{o}se, A proof of a conjecture on monotonic behavior of the largest eigenvalue of a number-theoretic matrix, 12th International Conference of Numerical Analysis and Applied Mathematics, Rhodes, Greece, 2014.

\bibitem{AltinisikBuyukkose2015} E. Alt\i n\i \c{s}\i k and \c{S}. B\"{u}y\"{u}kk\"{o}se, A proof of a conjecture on monotonic behavior of the smallest and the largest eigenvalue of a number-theoretic matrix, Linear Algebra Appl. 471 (2015) 141-149.

\bibitem{Altinisik2004} E. Alt\i n\i \c{s}\i k, N. Tu\~{g}lu, P. Haukkanen, A note on bounds for norms of the reciprocal LCM matrix,
Math. Inequal. Appl. 7.4 (2004) 491-496.

\bibitem{BeslinLigh} S. Beslin and S. Ligh, Greatest common divisor matrices,
Linear Algebra Appl. 118 (1989) 69-76.

\bibitem{BourqueLigh1992} K. Bourque and S. Ligh, On GCD and LCM matrices, Linear Algebra Appl. 174 (1992) 65-74.

\bibitem{Hauk1997} P. Haukkanen, J. Wang and J. Sillanp\"{a}\"{a}, On Smith's determinant,
Linear Algebra Appl. 258 (1997) 251-269.

\bibitem{Hong2008} S. Hong, Asymptotic behavior of largest eigenvalue of matrices associated with completely even functions (mod $r$) Asian-Europ.J. Math. 1 (2008) 225-235.

\bibitem{HongEnochLee} S. Hong and K. S. Enoch Lee, Asymptotic behavior of eigenvalues of reciprocal power LCM matrices, Glasg. Math. J.
50 (2008) 163-174.

\bibitem{HongLoewy2004} S. Hong and R. Loewy, Asymptotic behavior of eigenvalues of
greatest common divisor matrices, Glasg. Math. J. 46 (2004) 303-308.

\bibitem{HongLoewy2011} S. Hong and R. Loewy, Asymptotic behavior of the smallest eigenvalue of matrices associated
with completely even functions (mod $r$), Int. J. Number Theory 7 (2011)
1681-1704.

\bibitem{HornJohnson} R. Horn and C. R. Johnson, Matrix
Analysis, Cambridge University Press, Cambridge, London, 1985.

\bibitem{IlmonenHaukkanenMerikoski} P. Ilmonen, P. Haukkanen and J. K. Merikoski, On eigenvalues of meet and join matrices associated with incidence functions, Linear Algebra Appl. 429 (2008) 859-874.

\bibitem{Kalman} D. Kalman, A matrix proof of Newton's identities, Math. Mag. 73 (4) (2000) 859-874.

\bibitem{KorkeeHauk2003} I. Korkee and P. Haukkanen, On meet and join matrices associated with incidence functions, Linear Algebra Appl. 372 (2003) 127-153.

\bibitem{Lindqvist} P. Lindqvist and K. Seip, Note on some greatest common divisor matrices,
Acta Arith. 84.2 (1998) 149-154.

\bibitem{Mattila2015} M. Mattila, On the eigenvalues of combined meet and join matrices, Linear Algebra Appl. 466 (2015) 1-20.

\bibitem{Mattila2012} M. Mattila, P. Haukkanen, On the eigenvalues of certain number-theoretic matrices, International Conference in Number Theory and Applications 2012.

\bibitem{Mattila2012paper} M. Mattila, P. Haukkanen, On the eigenvalues of certain number-theoretic matrices. East-West J. Math. 14 (2012), no. 2, 121-130.

\bibitem{MattilaHauk2014} M. Mattila and P. Haukkanen, On the positive definiteness and eigenvalues of meet and join matrices, Discrete Math. 326 (2014) 9-19.

\bibitem{Smith} H. J. S. Smith, On the value of a certain artihmetical determinant,
Proc. London Math. Soc. Ser.1 7 (1876) 208-212.

\bibitem{SuliMayers} E. S\"{u}li and D. Mayers, An Introduction to Numerical Analysis, Cambridge University Press, Cambridge, London, 2003.

\bibitem{Wintner} A. Wintner, Diophantine approximations and Hilbert's space, Amer. J.
Math. 66 (1944) 564-578.

\end{thebibliography}
\end{document}